\documentclass[leqno,12pt]{article}

\usepackage{epsfig}

\usepackage{amsmath}
\usepackage{amscd}
\usepackage{amsopn}
\usepackage{amsthm}
\usepackage{amsfonts,amssymb}\usepackage{amsfonts, bbm}
\usepackage{latexsym}

\usepackage{color}

\makeatletter
\@addtoreset{equation}{section}
\makeatother

\newtheorem{lemma}{LEMMA}[section]
\newtheorem{proposition}[lemma]{PROPOSITION}
\newtheorem{corollary}[lemma]{COROLLARY}
\newtheorem{theorem}[lemma]{THEOREM}

\newtheorem{remarks}[lemma]{REMARKS}

\newcommand{\real} {\mathbbm{R}}

\newcommand{\nat}{\mathbbm{N}}
\newcommand{\rat}{\mathbbm{Q}}

\newcommand{\limn}{\lim_{n \to \infty}}

\newcommand{\vp}{\varphi}
\newcommand{\ve}{\varepsilon}

\newcommand{\reald}{{\real^d}}

\newcommand{\on}{\quad\text{ on }}
\newcommand{\und}{\quad\mbox{ and }\quad}

\newcommand{\ov}{\overline}

\newcommand{\V}{\mathcal V}  
\newcommand{\W}{\mathcal W}  
  
\newcommand{\C}{\mathcal C}  

\newcommand{\F}{\mathcal F}

\renewcommand{\H}{{\mathcal H}}
\newcommand{\B}{\mathcal B}
\newcommand{\J}{\mathcal J}
\renewcommand{\S}{\mathcal S}
\newcommand{\M}{\mathcal M}

\newcommand{\itemframe}%
{\setlength{\parskip}{10pt}\begin{enumerate} \setlength{\topsep}{10pt}%
\setlength{\itemsep}{15pt}\setlength{\parsep}{5pt}}

\newcommand{\vx}{\ve_x}

\date{}

\title{Reduced functions and Jensen measures} 
\author{Wolfhard Hansen and Ivan Netuka}

\begin{document}
\maketitle 

\begin{abstract}
Let $\varphi$ be a locally upper bounded Borel measurable function on a Greenian open set $\Omega$ in $\mathbbm R^d$
and, for every $x\in \Omega$, let $v_\varphi(x)$ denote the infimum of the integrals of $\varphi$ with respect to 
Jensen measures  for $x$ on $\Omega$. Twenty years ago, B.J.\,Cole and T.J.\,Ransford proved that
 $v_\varphi$ is the supremum of all subharmonic minorants of $\varphi$ on $X$ and that the sets
$\{v_\varphi<t\}$, $t\in \real$, are analytic. In this paper,  a different method leading to the inf-sup-result
establishes at the same time that, in fact,   
$v_\varphi$ is the minimum of $\varphi$ and a~subharmonic function, 
and hence Borel measurable. This is presented in the generality
of harmonic spaces, where semipolar sets are polar, and the key tools are   measurability results for 
reduced functions on balayage spaces which are of independent interest.


\end{abstract}

\section{Introduction}
The motivation for our considerations is a question in connection with Jensen measures
which could not be answered in \cite{cole-ransford-subharmonicity}. Let   $\Omega$
be an open set in $\reald$, $d\ge 2$ (such that, if $d=2$, $\reald \setminus \Omega$
is not polar). We recall that a~(Radon) measure $\mu$ with compact support in $\Omega$ is a \emph{Jensen measure
for a point $x\in \Omega$} if 
\begin{equation}\label{jensen-def}
\int v\,d\mu\ge v(x) \quad\mbox{ for every subharmonic function $v$ on $\Omega$}.
\end{equation} 
Let $\vp$ be a locally upper  bounded Borel measurable function on $\Omega$ and 
\begin{equation*}
        v_\vp(x):=\inf \{\int \vp\,d\mu\colon \mbox{$\mu$ a Jensen measure for $x$}\}, \qquad x\in \Omega.
\end{equation*} 
The results \cite[Theorem 1.6 and Corollary 1.7]{cole-ransford-subharmonicity} show that 
\begin{equation}\label{main-CR}
       v_\vp=\sup\{v\colon \mbox{ $v$ subharmonic on $\Omega$, } v\le \vp\}
\end{equation} 
and that the sets $\{v_\vp<t\}$, $t\in \real$, are analytic (which led the authors B.J.\,Cole and T.J.\,Ransford to a~definition
 and the study of quasi-subharmonic functions; cf.\  also~\cite{alakhrass-hansen}). It remained an open question if 
the function $v_{\vp}$ is, in fact,  Borel measurable (see the lines following \cite[Theorem 1.6]{cole-ransford-subharmonicity}).

In this short paper, we shall give a positive answer 
to this question (even in a much more general setting)
using a different method which, at the same time, provides a~simpler proof for (\ref{main-CR}). 

Our essential tools are measurability properties which we shall prove for reduced functions 
on balayage spaces $(X,\W)$ satisfying the axiom of polarity (Section 2) and which are of independent interest. 

In our application to Jensen measures on  harmonic spaces  (Section 3)
 it is natural to consider superharmonic functions instead of subharmonic functions. Recalling that a function $u$ is superharmonic 
if and only if $-u$ is subharmonic, this requires us to look upside-down at the  definitions,
assumptions and statements above. 

In both sections, the reader, who is not familiar with or not interested
in general potential theory, may suppose that $X$ is an open subset $\Omega$ of $\reald$   and that 
$\W$ is the set of all   functions $u\ge 0$ on $\Omega$ which are hyperharmonic on $\Omega$ (that is, which,
for each connected component $U$ of $\Omega$, are either superharmonic on~$U$ or are identically~$+\infty$ on~$U$).

\section{Measurability of reduced functions}

Let $(X,\W)$ be a balayage space 
 ($X$ a locally compact space with countable base   and~$\W$ the set of all hyperharmonic functions $u\ge 0$ on $X$,   
see \cite{BH} or \cite{H-course}).   In the following, let $u_0$
be any strictly positive function in $\W\cap\C(X)$ (say  $u_0=1$ if $1\in \W$). 
We denote by~$\B(X)$, $\C(X)$ respectively  the set of all numerical Borel measurable functions,
real continuous functions on $X$.   
As usual, given a set~$\F$ of functions, let~$\F^+$ 
be the set of all $f\in \F$ such that  $f\ge 0$.

We recall that, for every numerical  function $\vp\ge 0$ on $X$,   a \emph{reduced function}~$R_\vp$ is defined by
 \begin{equation}\label{def-red}
      R_\vp:=\inf\{u\in \W\colon u\ge \vp\}. 
\end{equation} 
It is easily seen that the mapping $\vp\mapsto R_\vp$ is subadditive, positively homogeneous, and $R_{R_\vp}=R_\vp$.
In particular, we have $R_v^A:=R_{v1_A}$ for $A\subset X$ and $v\in \W$, 
which leads to   reduced measures~$\vx^A$, $x\in X$, characterized by $\int v\,d\vx^A=R_v^A(x)$,
$v\in \W$ (by~\cite[VI.1.1]{BH}, the mappings $v\mapsto R_v^A$ are additive).

Let $\mathcal P(X)$ denote the set of all continuous real potentials on $X$,
that is,   of all $p\in\W\cap \C(X)$ satisfying
\begin{equation*}
   \inf \{R_p^{X\setminus K}\colon K\mbox{ compact in }X\}=0.
\end{equation*} 
A real function $\vp$ on $X$ is called $\mathcal P$-bounded, if $|\vp|\le p$ for some $p\in \mathcal P(X)$
(every bounded $\vp$ with compact support is $\mathcal P$-bounded).

For every numerical function $v$ on $X$, let $\hat v$ denote its lower semicontinuous regularization,
that is, 
\begin{equation*} 
              \hat v(x):=\liminf\nolimits_{y\to x} v(y), \qquad x\in X.
\end{equation*} 
If $\V\subset \W$ and $v:=\inf \V$,
then $\hat v\in \W$. So $\hat R_\vp:=\widehat{R_\vp}\in \W$ for every \hbox{$\vp\colon X\to [0,\infty]$}.
If~$\vp$ is lower semicontinuous, then $\hat R_\vp\ge \hat \vp=\vp$, and hence $\hat R_\vp\ge R_\vp$,
 $R_\vp=\hat R_\vp\in\W$. 
Moreover, $R_\vp$ is continuous, upper semicontinuous respectively,  
if $\vp$ is a $\mathcal P$-bounded function which is continuous,  upper semicontinuous (see  \cite[Corollary~1.2.2]{H-course}).

A subset $P$ of $X$ is \emph{polar}, if it has no interior points and, for every $x\in X$,
there exists a function $u\in\W$ such that $u=\infty$ on $P\setminus \{x\}$ and $u(x)<\infty$.
In particular, we know that $R_\vp=\vp$ and $\hat R_\vp=0$,  if $\vp$ vanishes outside a polar set.

\emph{Throughout this paper,  let us suppose that $(X,\W)$ satisfies the axiom of
  polarity {\rm(}Hunt's hypothesis {\rm(H)}{\rm)}: 
Every \emph{semipolar set}, that is, every set $\{\hat v<v\}$, where $v=\inf \V$
for some $\V\subset \W$, is polar.}

Let $\tilde \B(X)$ denote the set of all numerical functions $\vp$ on $X$ 
for which there exist functions $\vp_1,\vp_2\in\B(X)$ 
with $\vp_1\le \vp\le \vp_2$ on $X$   and $\vp_1=\vp_2$ outside a polar set.

\begin{proposition}\label{hat-R-R}
Let  $\vp\colon X\to [0,\infty] $ and  $P$ be a polar set such that   
$\hat R_\vp=R_\vp$ on $X\setminus P$.  Then 
\begin{equation}\label{mertens-id} 
    R_\vp=(\vp 1_P)\vee   \hat R_{\vp 1_{X\setminus P}} =\vp\vee \hat R_\vp  \und 
\hat R_{\vp1_{X\setminus P}}= R_{\vp1_{X\setminus P}}=\hat R_\vp.
\end{equation} 
In particular, $R_\vp\in \tilde \B(X)$.  Moreover, $R_\vp\in   \B(X)$ if $\vp\in\B(X)$. 
\end{proposition} 

\begin{proof}  Let $f:=\vp 1_P$ and  $g:=\vp 1_{X\setminus P}$. Trivially, 
\begin{equation}\label{trivial-Rvp} 
f\vee \hat R_g\le \vp\vee \hat R_\vp\le R_\vp.
\end{equation} 
Moreover, $\hat R_\vp\in\W$ and  $\hat R_\vp=R_\vp\ge g$ on $X\setminus P$. Therefore 
\begin{equation}\label{hatg}
\hat R_\vp\ge R_g\ge \hat R_g.
\end{equation} 
Defining   $f_0:=1_{\{R_g<\infty\}}(f-R_g)^+$ we have  $\vp=f\vee g\le f\vee R_g=f_0+R_g$. Further,
$R_{f_0}=f_0$, since $f_0=0$ outside the polar set $P$, and we obtain that 
\begin{equation*}
                    f\vee R_g\le R_{f\vee g}=R_\vp\le R_{f_0+R_g}\le R_{f_0}+R_{R_g}=f_0+R_g=f\vee R_g.
\end{equation*} 
So  $R_\vp=f\vee R_g=f_0+R_g$. In particular, 
$\hat R_\vp=\hat R_g$ on the complement of  the (semi)polar set $P\cup\{\hat R_g<R_g\}$, and hence $\hat R_\vp=\hat R_g$
(see \cite[VI.5.10]{BH}).
  Having~(\ref{hatg}) the second part of (\ref{mertens-id}) follows. Its first part is now an immediate
consequence of~(\ref{trivial-Rvp}) and the equality $R_\vp=f\vee R_g$.
\end{proof} 

Combining Proposition  \ref{hat-R-R} with the fact (see \cite[VI.1.9]{BH}) that,
 for every Borel set $A$, the function $\hat R_{u_0}^A$is  the supremum of all $\hat R_{u_0}^K$, 
$K$ compact in $A$, we then obtain the following result.

\begin{theorem}\label{usemi}
Let $\vp\in \tilde \B^+(X)$  and let $\Psi$ denote the set of all  
 bounded upper semicontinuous functions $\psi\ge 0$ with compact support
in $\{\vp>0\}$. Then there exists an increasing sequence $(\psi_n)$ in $\Psi$  such that 
\begin{equation}\label{form-usemi}
\hat R_\vp=   \sup\nolimits_{n\in \nat} \hat R_{\psi_n}.
\end{equation} 
In particular,
\begin{equation}\label{final}
            R_\vp= \vp\vee  \sup\nolimits_{n\in \nat} \hat R_{\psi_n}
=\vp\vee  \sup\nolimits_{n\in \nat}   R_{\psi_n}=\sup\{R_\psi\colon \psi\in \Psi\}.
\end{equation} 
\end{theorem}

\begin{proof}  Let $\vp_1,\vp_2\in\B(X)$ such that $\vp_1\le \vp\le \vp_2$ and   $P_0:=\{\vp_1\ne \vp_2\}$ is polar.  
For every $t\in\rat^+$, let $A_t$ be the Borel subset $\{\vp_1>tu_0\}$ of $\{\vp>tu_0\}$. The union~$P$ of~$P_0$,
the set $\{\hat R_\vp<R_\vp\}$, and the sets   $\{\hat R_{u_0}^{A_t}<R_{u_0}^{A_t}\}$, $t\in\rat^+$, is polar. 

Let $x\in X\setminus P$ and  $a<\vp(x)$. Let us choose  $t\in\rat^+$ such that   $a<tu_0(x)<\vp(x)$.
Then $x\in A_t$  and   $ \hat R_{u_0}^{A_t}(x) = R_{u_0}^{A_t}(x) =  u_0(x)>a/t$.  
Hence $ \hat R_{u_0}^K(x)>a/t$ for some  compact $K$ in~$A_t$. Obviously, $\psi:=tu_01_K\in\Psi$ and  
$\hat R_\psi(x)=t\hat R_{u_0}^K(x)>a$. 

This shows  that  
\begin{equation*}
               u:=     \sup\{\hat R_\psi\colon \psi\in \Psi\}\ge \vp \on X\setminus P.
\end{equation*} 

If $\psi_1,\psi_2\in \Psi$, then $\psi:=\psi_1\vee \psi_2\in \Psi$ and
 $\hat R_{\psi_1}\vee \hat R_{\psi_2}\le \hat R_\psi$.
So,  by \cite[I.1.7]{BH},  there exists  an increasing sequence $(\psi_n)$ in $\Psi$ 
with $u=\sup_{n\in\nat}  \hat R_{\psi_n}$. 
 In particular, $u\in\W$, and hence $u\ge R_{\vp1_{X\setminus P}} $.   
 Since trivially $u\le \hat R_\vp$,
the proof is completed by  Proposition \ref {hat-R-R}, monotonicity, 
and the  
 fact that $\vp1_{\{x\}} \in\Psi$ for \hbox{$x\in\{\vp>0\}$}. 
\end{proof}

\section{Application to Jensen measures}

From now on, we suppose more restrictively that the balayage space
$(X,\W)$ satisfying the axiom of polarity is a harmonic space, that is, 
  $\W$ has the following local truncation property: For all open sets~$U$
in~$X$ and all $u,v\in \W$ such that $u\ge v$ on the boundary~$\partial U$ of~$U$, the
function~$w$ defined by $w:=u\wedge v$ on~$U$ and $v$ on~$X\setminus U$
is contained in~$\W$ (see \cite[Section III.8]{BH}). This means that 
the reduced measures~$\vx^{X\setminus V}$ (that is, the harmonic     
measures~$\mu_x^V$) 
for open sets~$V$ and~$x\in V$ are supported by~$\partial V$ (instead of having
supports which could be the entire complement of $V$).

  In probabilistic terms,
an associated process will be a diffusion (instead of a~process possibly having many jumps).
We recall that fairly general linear differential operators $L$ of second order on open subsets~$X$ of~$\reald$
($L$ being the Laplacian in the classical case) lead to  harmonic spaces (see, for example, \cite[Section 7]{GH1}).

Given an open set $U$ in~$X$, let ${}^\ast\H(U)$ denote the set of all \emph{hyperharmonic functions}~$v$
on~$U$, that is, of all lower semicontinuous $v\colon U\to \left]-\infty,\infty\right]$ such that
$\int v \,d\mu_x^V\le v(x)$ for every open set $V$, which is relatively compact in $U$,  and every~$x\in V$. 
If, in addition, the functions $x\mapsto \int v\,d\mu_x^V$ are continuous and finite on~$V$, then such a function~$v$
is called \emph{superharmonic} on~$U$. The set of all superharmonic functions on~$U$ is denoted by~$\S(U)$,
and $\H(U)=\S(U)\cap (-\S(U))$ is the set of all harmonic functions on $U$.

We note that ${}^\ast\H^+(X)=\W$ and $\S^+(X)\cap \C(X)=\W\cap \C(X)$.
In particular, it is compatible with (\ref{def-red}) to define, for \emph{every} numerical
function~$\vp$ on~$X$, 
\begin{equation*}
      R_\vp:=\inf\{ v \in {}^\ast\H(X)\colon v\ge \vp\}.
\end{equation*} 

In our proofs we shall tacitly use that, for every (relatively compact) open set~$U$ in~$X$,
$(U,{}^\ast \H^+(U))$ is a harmonic space as well (see \cite [V.1.1]{BH} in connection with     
\cite[III.2.8 and 6.11]{BH})   
and that sets $A\subset U$ which are polar (semipolar, respectively) with respect to~$(U,{}^\ast \H^+(U))$
are polar (semipolar, respectively) with respect to $(X,\W)$ (see \cite[Sections~6.2 and~6.3]{Const}; the converse is trivial).

Given an open set $U$ in $X$, we say that  a locally lower bounded function~$v$ on~$U$
 is \emph{nearly hyperharmonic} if
$\int^\ast v\,d\mu_x^V\le v(x)$ for every open set $V$, which is relatively compact in $U$,
and every $x\in V$. As is well-known, $\hat v\in {}^\ast\H(U)$ for every nearly hyperharmonic function on $U$.

\begin{lemma}\label{nearly-hyper} Let  $v$   be a locally lower bounded numerical function on an open set~$U$ in~$X$.
The  following statements are equivalent:
\begin{itemize}
\item[\rm (i)]
$v$ is nearly hyperharmonic on $U$ and the set $\{\hat v<v\}$ is polar.
\item[\rm (ii)] $v$ is the infimum of its hyperharmonic majorants on $U$. 
\end{itemize} 
\end{lemma} 

\begin{proof} 
If (i) holds, we may argue as in the proof of (1)\,$\Rightarrow$\,(2) in \cite[Theorem 2]{alakhrass-hansen}):
Let $x\in U$ be such that $v(x)<\infty$, and let $\ve>0$. There exists $v_x\in {}^\ast\H^+(U)$ such that 
$v_x(x)=v(x)-\hat v(x)+\ve$ and $v_x=\infty$ on the polar set $\{\hat v<v\} \setminus \{x\}$. 
Then $w:=\hat v+v_x\in {}^\ast\H^+(U)$, $w\ge v$ and $w(x)=v(x)+\ve$.

Next suppose that (ii) holds. Then $v$ is obviously nearly hyperharmonic on $U$. Moreover, the set $\{\hat v<v\} $
is semipolar (see \cite[Theorem 6.3.2]{Const}), and hence polar by the axiom of polarity.
\end{proof} 

We shall use the following consequence. 

\begin{lemma}\label{crucial}
Let $U_n$, $n\in\nat$, be relatively compact open sets in $X$ such that $\ov U_n\subset U_{n+1}$ and
$\bigcup_{n\in\nat} U_n=X$. Moreover, let $(v_n)$  be an increasing sequence of locally lower bounded 
numerical functions on $X$   such that, for every $n\in\nat$,
\begin{equation*}
                                 v_n|_{U_n}=\inf \{w\in {}^\ast \H(U_n)\colon w\ge v_n|_{U_n}\},
\end{equation*} 
and let  $v:=\limn v_n$. Then $\hat v=\limn \hat v_n$ and 
\begin{equation}\label{vRv} 
v=\inf\{w\in {}^\ast\H(X)\colon w\ge v\}.
\end{equation} 
\end{lemma} 

\begin{proof} 
For every $n\in\nat$,
$v_n$ is nearly hyperharmonic on $U_n$  and   $P_n:=\{\hat v_n<v_n\}$ is polar,
by Lemma \ref{nearly-hyper}. Therefore $v$ is nearly hyperharmonic on $X$
and   $\hat v=\limn\hat v_n$  (see \cite[p.\,48]{bauer66}). Hence the set 
$P:=\{\hat v<v\}$ is contained in the union of all $P_n$, $n\in\nat$. So $P$ is polar,
and (\ref{vRv}) holds, by Lemma \ref{nearly-hyper}. 
\end{proof} 

For every open set $U$ in $X$, let $\M_c(U)$ denote the set of all measures with compact   support in~$U$. 
For every  $x\in U$, let $\J_x(U)$ denote the set of all \emph{Jensen measures for~$x$ with
  respect to~$U$}, that is,
\begin{equation*}
  \J_x(U):=\{\mu\in \M_c(U)\colon \int v \,d\mu\le v(x)\mbox{ for every }
  v\in \S(U)\}. 
\end{equation*} 

If $h\in\H(U)$, then $\pm h\in  \S(U) $, and hence
\begin{equation*}
\int h\,d\mu=h(x) \quad\mbox{  for all   $x\in U$ and $\mu\in \J_x(U)$}.
\end{equation*} 
Since every function in  $ {}^\ast\H(U)$  is an increasing limit
of functions in $\S(U)\cap\C(U)$ (see \cite[Corollary 2.3.1]{Const}), a measure
$\mu\in \M_c(U)$ is a Jensen measure for $x$ with respect to $U$
provided $\int u\,d\mu\le u(x)$ for every $u\in \S(U)\cap \C(U)$, and then
$\int w\,d\mu\le w(x)$ for every $w\in{}^\ast \H(U)$.

Of course, $\J_x(U)$ is a convex set containing the Dirac measure $\ve_x$ at $x$ 
and the harmonic measures $\mu_x^V$,
$V$ relatively compact   open in $U$ and $x\in V$  (see \cite{HN-jensen} for a~detailed discussion).

Let $\vp\in \tilde \B(X)$ be locally lower  bounded.  
If $x\in X$, $\mu\in \J_x(X)$, 
then $\mu^\ast(P)=0$ for every polar set $P\subset X\setminus \{x\}$ (if $u\in\W$ such that $u=\infty$ 
on $P\setminus \{x\}$ 
and $u(x)<\infty$, then $\infty\cdot \mu^\ast(P)\le \int u\,d\mu\le u(x)<\infty$). 
Hence we may define a~function $J_\vp$ on $X$ by 
\begin{equation}\label{def-u_}
   J_\vp(x):=\sup\{\int \vp\,d\mu\colon \mu \in \J_x(X)\}, \qquad x\in
   X.
\end{equation} 
Trivially,
\begin{equation}\label{vp-trivial}
                 J_\vp \le \inf\{w\in {}^\ast \H(X)\colon w\ge \vp\}=R_\vp.
\end{equation} 
Let us begin by proving the reverse inequality for $\vp\ge 0$ 
 (see Theorem \ref{u-loc-theorem} for the general case).
A first step is the following.

\begin{proposition}\label{R-psi}
  Let $\psi\ge 0$ be a  $\mathcal P$-bounded upper 
 semicontinuous function on~$X$. Then $  J_\psi=R_\psi$. 
In particular, $J_\psi$ is upper semicontinuous.
\end{proposition} 

\begin{proof} 
Let us fix  an exhaustion of $X$ by relatively compact open
sets~$U_n$, $n\in\nat$, such that  $\ov U_n\subset U_{n+1}$. For $n\in\nat$, we define a function $v_n\ge \psi$ 
on $X$ by 
\begin{equation*} 
v_n(x): = \inf \{s(x)\colon s\in \S(X)\cap \C(X), \, s\ge \psi \mbox{  on }\ov U_n\}, 
\qquad x\in U_n,
\end{equation*} 
and $v_n:=\psi$ on $X\setminus   U_n$. 
Of course, $v_n|_{U_n}=\inf\{w\in {}^\ast\H(U_n)\colon w\ge v_n|_{U_n}\}$, $n\in\nat$,  and the sequence $(v_n)$
is increasing. By Lemma \ref{crucial}, 
$                          v:=\limn  v_n
$
satisfies $v=R_v$. Since $v\ge \psi$, we see that $v\ge R_\psi$.

Let $(\vp_m)$ be a  sequence of continuous $\mathcal P$-bounded functions 
which is decreasing to $\psi$. 
Let us fix $x\in X$ and consider $n\in\nat$
with $x\in   U_n$. 
By the theorem of Hahn-Banach, there are measures $\nu_m\in \J_x(X)$, $m\in \nat$,  which are supported by~$\ov U_n$
and satisfy 
\begin{equation*} 
                                                \int \vp_m\,d\nu_m=\inf \{s(x)\colon s\in \S(X)\cap \C(X), \, s\ge \vp_m\mbox{  on }\ov U_n\}
\end{equation*} 
(see, for example, \cite[I.2.3]{BH}) so that obviously $ \int \vp_m\,d\nu_m\ge v_n(x)$. 
Having the inequalities  $\int u_0\,d\nu_m\le u_0(x)$ we know  that $\nu_m(\ov U_n)\le u_0(x)/\inf u_0(\ov U_n)$ for all $m\in\nat$.

 Passing to a subsequence we hence may  assume without loss
of generality that the sequence $(\nu_m)$ converges weakly to a measure $\nu$ on $\ov U_n$ (that is,
$\lim_{m\to\infty} \nu_m(f)=\nu(f)$ for every $f\in \C(\ov U_n)$). Then, of course, $\nu\in\J_x(X)$ and, for every $k\in\nat$,
\begin{equation*}
v_n(x) \le  \liminf_{m\to\infty} \int \vp_m\,d\nu_m
\le \lim_{m\to\infty} \int \vp_k\,d\nu_m=\int\vp_k\,d\nu.
\end{equation*} 
Letting $k\to \infty$, we see that $v_n(x)\le \int \psi\,d\nu\le J_\psi(x)$.  
We finally let $n\to\infty$ and, using (\ref{vp-trivial}) and $R_\psi\le v$, obtain that 
$v(x)=J_\psi(x)=R_\psi(x)$. 
\end{proof}

Having Theorem \ref{usemi}, an easy consequence is the following.

\begin{corollary}\label{vp-positive} 
Let $\vp\in\tilde \B(X)$ and  $\vp+h\ge 0$ for some $h\in\H(X)$. 
Then 
\begin{equation*}
J_\vp= R_\vp=\vp\vee \hat R_\vp.
\end{equation*}  
\end{corollary} 

\begin{proof} 
(a) Let us suppose first that $\vp\ge 0$.
By (\ref{vp-trivial}), $J_\vp\le R_\vp$.
On the other hand, by Theorem \ref{usemi},
there exist bounded upper semicontinuous functions $\psi_n$ with compact support  
 which  satisfy $0\le \psi_n\le \psi_{n+1}\le \vp$, $n\in\nat$, and 
\begin{equation*} 
 R_\vp=                          \vp\vee   \sup\nolimits_{n\in\nat} R_{\psi_n}.
\end{equation*} 
Since $\ve_x\in \J_x(X)$ for every $x\in X$, we know that $\vp\le
J_\vp$.  By Proposition~\ref{R-psi},
$R_{\psi_n}=J_{\psi_n}\le J_\vp$ for all $n\in\nat$. Thus also $R_\vp\le J_\vp$.
By Theorem \ref{hat-R-R}, $R_\vp=\vp\vee \hat R_\vp$.

(b) In the general case $\vp+h\ge 0$ it suffices to observe that $\vp+h\in \tilde \B(X)$,  
 hence $J_{\vp+h}=R_{\vp+h}$, by (a), and that obviously     
$    J_\vp=J_{\vp+h}-h$ and $R_\vp=R_{\vp+h}-h$. 
   \end{proof} 

To obtain the same result  for 
functions $\vp\in \B(X)$ which are only supposed to be locally lower bounded, 
we shall apply  Corollary \ref{vp-positive} to relatively compact open subsets~$U$ of~$X$   
assuming that on these sets $U$ there exist   strictly positive harmonic functions.  
This is a rather weak assumption; it  is equivalent to $R_{u_0}^{X\setminus U}>0$.
In this process, we have to work with the  subset
$\J_x'(X)$ of $\J_x(X)$,  $x\in X$, defined by
\begin{equation*}
          \J_x'(X):=\{\mu\in \M_c(X)\colon \mu \in
          \J_x(U) \mbox{ for some relatively
            compact open $U$ in }X\},
\end{equation*} 
and to consider also functions  $J_\vp'$ defined by
\begin{equation*}
      J_\vp'(x):=\sup\{\int \vp\, d\mu\colon \mu\in \J_x'(X)\}.
\end{equation*} 
For the sake of completeness, we   recall from \cite{HN-jensen} that
fairly weak assumptions on~$(X,\W)$ imply that  
$\J_x'(X)=\J_x(X)$ for every $x\in X$ (see Remark \ref{J-prime-J},2).     

Here is the main result in this Section.

\begin{theorem}\label{u-loc-theorem}
Let $\vp\in \tilde \B(X)$ be locally lower bounded. Then 
\begin{equation*}
J_\vp=J_\vp'=R_\vp=\vp\vee \hat R_\vp.
\end{equation*} 
In particular, $J_\vp\in \tilde \B(X)$. Moreover, $J_\vp\in \B(X)$ if $\vp\in\B(X)$.   
\end{theorem} 

\begin{proof} 
Since $\J_x'(X)\subset \J_x(X)$, $x\in X$, and (\ref{vp-trivial}) holds,
we have the inequalities  
\begin{equation*} 
     R_\vp\ge J_{\vp}\ge J_\vp'.
\end{equation*} 

To prove that $ J_\vp'\ge R_\vp$
let us choose again relatively compact open sets $U_n$ exhausting $X$
such that $\ov U_n\subset U_{n+1}$ for every $n\in \nat$.
For the moment,  let us fix $n\in\nat$.  By assumption, there is
a strictly positive function $h_{n+1}\in \H(U_{n+1})$, and  
  there exists $a_n>0$ such that the function $h_n:=a_n h_{n+1}|_{U_n}\in \H^+(U_n)$
satisfies $\vp+h_n>0$ on~$U_n$.
 By Corollary  \ref{vp-positive} (applied to $U_n$ instead of $X$),  
\begin{equation}\label{vnvp}
v_n:= \inf\{  w\in {}^\ast\H(U_n) \colon w\ge \vp \mbox{ on }U_n\} = (\vp|_{U_n})\vee \hat v_n
 \end{equation} 
and, for every $x\in U_n$,
\begin{equation}\label{vnUn} 
v_n(x)= \sup\{\int \vp\,d\mu\colon \mu\in \J_x(U_n)\}. 
\end{equation} 
Extending the functions $v_n$ to functions on $X$ by $v_n(x):=\vp(x)$, $x\in X\setminus U_n$,
(\ref{vnUn}) implies that the sequence $(v_n)$  is increasing to $v:=J_\vp'$. By Lemma \ref{crucial},
we   conclude that $v=R_v$ and $\hat v=\limn \hat v_n$.   Since   $v\ge \vp$, we obtain that $J_\vp'=v\ge R_\vp$.

Thus $J_\vp=J_\vp'=R_\vp$, and we finally see that $R_\vp=\vp\vee \hat R_\vp$, by (\ref{vnvp}).
\end{proof}

\begin{corollary}\label{char-inf} For every locally lower bounded numerical function $u$ on $X$
the following three statements are equivalent:
\begin{itemize}
\item[\rm (i)] $u$ is the infimum of its hyperharmonic majorants.
\item[\rm (ii)] $u\in \tilde \B(X)$ and $\int u\,d\mu\le u(x)$ for all $x\in X$ and $\mu\in \J_x(X)$.
\item[\rm (iii)] $u\in \tilde \B(X)$ and $\int u\,d\mu\le u(x)$ for all $x\in X$ and $\mu\in \J_x'(X)$.
\end{itemize} 
\end{corollary} 

\begin{proof} Having Theorem \ref{u-loc-theorem} it suffices to observe that (iii) implies  $J'_u=u$.
\end{proof}

\begin{remarks}\label{J-prime-J}
{\rm
1.
 An equivalence as in Corollary \ref{char-inf} is contained in (\cite[Theorem~2]{alakhrass-hansen} 
 under the stronger assumption of having a~Brelot space satisfying the  axiom of domination.

2. The detailed description of Jensen measures in \cite{HN-jensen} led to 
various simple properties implying that (without assuming the axiom of polarity)
\begin{equation}\label{JJ-id} 
 \J_x'(X)=\J_x(X) \qquad \mbox{ for every }x\in X. 
 \end{equation} 

 For example, (\ref{JJ-id}) holds
if $(X,\W)$ has the following approximation property (AP):
For every compact $K$ in $X$, there exists a relatively compact open neighborhood $U$
 of $K$ such that, for all $u\in \S(U)\cap \C(U)$ and $\ve>0$, there exists a function   
$v\in \S(X)\cap \C(X)$ satisfying $|u-v|<\ve$ on $K$.

If~$(X,\W)$ is elliptic, that is, if every     
superharmonic function $s\ge 0$, $s\ne 0$, 
 on a~domain $U$ in $X$
 is strictly positive, (AP) follows from  \cite[Theorem~6.1 and Remark~6.2.1]{BH-simplicial-II}
(cf.\ also \hbox{\cite[Theorem 6.9]{gardiner-app}} for
the classical case and \cite[Theorem~1]{{gardiner-gowri}} 
for the case of  a~Brelot space satisfying the axiom of domination).

An approach to (\ref{JJ-id}), which is much less involved
and, by \cite[Proposition 3.2]{HN-jensen},  covers the classical case as well, 
assumes that $(X,\W)$ is $h_0$-transient for some strictly positive $h_0\in \H(X)$, that is,
for every compact $K$ in $X$, the (closed) set $\{R_{h_0}^K=h_0\}$ is compact (\cite[Theorem 3.3]{HN-jensen},
see also \cite[Corollary 4.4]{HN-jensen} for several characterizations of $1$-transient bounded open sets in the classical case). 
}
\end{remarks}
% \bibliography{jensen_lit_bank}

\begin{thebibliography}{1}

\bibitem{alakhrass-hansen}
M.\,Alakhrass and W.\,Hansen.
\newblock Infima of superharmonic functions.
\newblock{\em Ark. Mat.}, 50:231--235, 2012.

\bibitem{bauer66}
 H.\,Bauer, 
\newblock{\em {Harmonische {R}\"aume und ihre {P}otentialtheorie}}.
 \newblock {Springer}, Berlin, 1966.

\bibitem{BH-simplicial-II}
J.~Bliedtner and W.~Hansen.
\newblock Simplicial cones in potential theory. {II}. {A}pproximation theorems.
\newblock {\em Invent. Math.}, 46(3):255--275, 1978.


\bibitem{BH}
J.\,Bliedtner and W.\,Hansen.
\newblock {\em {Potential Theory -- An Analytic and Probabilistic Approach to
  Balayage}}.
\newblock Universitext. Springer, Berlin-Heidelberg-New York-Tokyo, 1986.

\bibitem{cole-ransford-subharmonicity}
B.J.\, Cole and T.J.\, Ransford.
\newblock Subharmonicity without semicontinuity.
\newblock {\em J. Funct. Anal.}, 147:420--442, 1997.

 

\bibitem{Const}
C.\,Constantinescu and A.\,Cornea.
\newblock {\em {Potential Theory on Harmonic Spaces}}.
\newblock {Grundlehren d. math. Wiss.} Springer, Berlin - Heidelberg - New
  York, 1972.

\bibitem{gardiner-app}
S.~J. Gardiner.
\newblock {\em Harmonic approximation}, volume 221 of {\em London Mathematical
  Society Lecture Note Series}.
\newblock Cambridge University Press, Cambridge, 1995.


\bibitem{gardiner-gowri}
S.J. Gardiner, M.~Goldstein, and K.~GowriSankaran.
\newblock Global approximation in harmonic spaces.
\newblock {\em Proc. Amer. Math. Soc.}, 122(1):213--221, 1994.

\bibitem{GH1}
 A.\,Grigor'yan and W.\,Hansen.
\newblock {A Liouville property for Schr\"{o}dinger operators}.
 \newblock{\em Math. Ann.} 312: 659--716, 1998.

\bibitem{H-course}
W.\,Hansen.
\newblock{\em Three views on potential theory}.
\newblock A course at Charles University (Prague), Spring 2008.
\newblock http://www.karlin.mff.cuni.cz/\verb+~+hansen/lecture/course-07012009.pdf.  

\bibitem{HN-jensen}
W.\,Hansen and I.\,Netuka.
\newblock Jensen measures in potential theory.
\newblock {\em Potential Anal.}, 37:79--90, 2012.


\end{thebibliography}
%\bibliographystyle{plain}

{\small \noindent 
Wolfhard Hansen,
Fakult\"at f\"ur Mathematik,
Universit\"at Bielefeld,
33501 Bielefeld, Germany, e-mail:
 hansen$@$math.uni-bielefeld.de}\\
{\small \noindent Ivan Netuka,
Charles University,
Faculty of Mathematics and Physics,
Mathematical Institute,
 Sokolovsk\'a 83,
 186 75 Praha 8, Czech Republic, email:
netuka@karlin.mff.cuni.cz}

\end{document}